\documentclass[a4paper, oneside]{amsart}

\RequirePackage[l2tabu, orthodox]{nag}
\usepackage[all, warning]{onlyamsmath}
\usepackage{color}
\usepackage[new]{old-arrows}

\usepackage{amsmath}
\usepackage{amsthm}
\usepackage{amsfonts}
\usepackage{amssymb}
\usepackage{amscd}

\usepackage{mathrsfs}
\usepackage{bm}
\usepackage{bbm}
\usepackage{stmaryrd}
\usepackage{braket}

\usepackage[all]{xy}
\usepackage{pb-diagram}
\usepackage{graphicx}

\usepackage{url}
\usepackage[setpagesize=false, pagebackref]{hyperref}
\usepackage{mathtools}
\usepackage{hhline}

\hypersetup{
	colorlinks=true,       
	linkcolor=blue,          
	citecolor=blue,        
	filecolor=blue,      
	urlcolor=blue,           
}

\numberwithin{equation}{section}
\newtheorem{thm}{Theorem}
\numberwithin{thm}{section}
\newtheorem{lem}[thm]{Lemma}
\newtheorem{cor}[thm]{Corollary}
\newtheorem{prop}[thm]{Proposition}

\newtheorem{dfn}[thm]{Definition}

\theoremstyle{definition}
\newtheorem{rmk}[thm]{Remark}

\newcommand{\bZ}{\mathbb{Z}}
\newcommand{\Z}{\mathbb{Z}}
\newcommand{\Q}{\mathbb{Q}}

\newcommand{\F}{\mathbb{F}}

\newcommand{\bQ}{\mathbb{Q}}

\newcommand{\bbH}{\mathbb{H}}

\newcommand{\mfZ}{\mathfrak{Z}}

\newcommand{\mfz}{\mathfrak{z}}

\newcommand{\fz}{\mathfrak{z}}
\newcommand{\fZ}{\mathfrak{Z}}

\DeclareMathOperator{\Hom}{Hom}

\DeclareMathOperator{\Res}{Res}
\DeclareMathOperator{\Cor}{Cor}

\DeclareMathOperator{\sgn}{sgn}

\DeclareMathOperator{\N}{N}

\newcommand{\pr}{\mathrm{pr}}

\newcommand{\SL}{\mathrm{SL}}
\newcommand{\PSL}{\mathrm{PSL}}

\newcommand{\Sym}{\mathrm{Sym}}

\newcommand{\bs}{\backslash}

\newcommand{\ra}{\rightarrow}

\newcommand{\setm}{\!-\!}

\def\={\;=\;}
\def\+{\,+\,}
\def\-{\,-\,}
\def\:{\;\colon\;} 

\def\ceq{\,\coloneqq\,}

\newcommand{\ord}{\mathrm{ord}}

\newcommand{\pord}{p\text{-}\mathrm{ord}}

\newcommand{\pmat}[1]{\begin{pmatrix}#1\end{pmatrix}}
\newcommand{\smat}[1]{\big(\begin{smallmatrix}#1\end{smallmatrix}\big)}

\newcommand{\trace}{\mathrm{trace}}
\newcommand{\num}{\mathrm{numerator}}
\newcommand{\hyp}{\mathrm{hyp}}

\numberwithin{equation}{section}

\newcommand{\brk}[1]{\langle#1\rangle}

\newcommand{\Eis}{\mathrm{Eis}}

\newcommand{\LS}{\mathcal{V}}

\allowdisplaybreaks[3]

\begin{document}
\title{$p$-Ordinary Part of Hyperbolic Cycles on Modular Curves}


\author{Hohto Bekki}
\address{Hohto Bekki \newline 
Department of Mathematics\\
Saga University\\
1 Honjo-machi\\
Saga-shi\\Saga\\840-8502\\
Japan}
\email{bekki@cc.saga-u.ac.jp}

\author{Ryotaro Sakamoto}
\address{Ryotaro Sakamoto \newline Department of Mathematics\\University of Tsukuba\\1-1-1 Tennodai\\Tsukuba\\Ibaraki 305-8571\\Japan}
\email{rsakamoto@math.tsukuba.ac.jp}

\begin{abstract}
In this paper, we study hyperbolic cycles in the first homology group of congruence subgroups of $\SL_2(\Z)$ with local coefficients. 
We prove that, for any prime number $p$, the $p$-ordinary part of the first 
homology group is generated by hyperbolic cycles. 
\end{abstract}

\maketitle

\section{Introduction}\label{sec:Introduction}
To any hyperbolic matrix $\gamma = \pmat{a&b\\c&d} \in \SL_2(\Z)$, i.e., $|\trace(\gamma)|>2$, we associate a binary quadratic form $Q_{\gamma}(X_1,X_2) \in \bZ[X_1, X_2]$ by 
\begin{align*}
Q_{\gamma}(X_1,X_2) \ceq -\frac{\sgn(a+d)}{\gcd(c,a-d,b)}(cX_1^2-(a-d)X_1X_2-bX_2^2). 
\end{align*} 
By construction, the coefficients of $Q_\gamma(X_1,X_2)$ are coprime.
Moreover, $Q_\gamma(X_1, X_2)$ is $\gamma$-invariant, in the sense that
\[
Q_\gamma((X_1,X_2)\cdot {}^{t}\gamma^{-1})
=
Q_\gamma(X_1,X_2).
\]
Since $\gamma$ is hyperbolic, the binary quadratic form $Q_\gamma(X_1, X_2) \in \bZ[X_1, X_2]$ is uniquely determined by these two properties, up to sign. 
Using $Q_\gamma(X_1, X_2)$, for any congruence subgroup 
$\Gamma \subset \SL_2(\Z)$ and any integer $k \geq 0$, 
we canonically attach to a hyperbolic matrix $\gamma \in \Gamma$ 
a homology class
\[
\mfz_{\Gamma,2k}(\gamma)
\coloneqq
(\gamma - 1) \otimes Q_\gamma(X_1,X_2)^k
\in
H_1(\overline{\Gamma}, \LS_{2k}),
\]
where $\overline{\Gamma}$ denotes the image of $\Gamma$ in 
$\PSL_2(\Z)$ and $\LS_{2k} \coloneqq \Sym^{2k}(\Hom(\Z^2,\Z))$. 
We regard $\LS_{2k}$ as the space of homogeneous polynomials
of degree $2k$ in two variables $X_1$ and $X_2$ over $\Z$. 
Geometrically, the homology class $\mfz_{\Gamma,2k}(\gamma)$
is represented by the closed geodesic on the modular curve
$\Gamma \bs \bbH$ associated with $\gamma$ and carrying the
coefficient $Q_\gamma(X_1,X_2)^k$. 
We call $\mfz_{\Gamma,2k}(\gamma)$ a hyperbolic cycle for $\Gamma$ of degree $2k$ associated with $\gamma$, and we let
\[
  \mfZ_{\Gamma,2k}
  \coloneqq 
  \bigl\langle\, 
    \mfz_{\Gamma,2k}(\gamma)
    \mid 
    \gamma \in \Gamma, \, |\trace(\gamma)|>2 \, 
  \bigr\rangle_{\Z}
\]
denote the subgroup of $H_1(\overline{\Gamma},\LS_{2k})$ generated by the hyperbolic cycles for $\Gamma$ of degree $2k$.

In this paper, we study the $p$-ordinary part of the subgroup $\mfZ_{\Gamma,2k} \subset H_1(\overline{\Gamma},\LS_{2k})$ for every prime number $p$. 
The motivation for the present study comes from our previous work~\cite{BS25}, 
where we investigated the denominators of special values of partial zeta 
functions of real quadratic fields using Harder's theorem 
on the denominator of the Eisenstein class for $\SL_2(\Z)$ (see \cite[Theorem 2.13]{BS25} and \cite{Harder2025}). 
Since the pairing of hyperbolic cycles for $\SL_2(\Z)$ with the Eisenstein 
class gives rise to special values of partial zeta functions of real 
quadratic fields, we were naturally led to show that the subgroup 
$\mfZ_{\SL_2(\Z),2k}$ is ``sufficiently large''. 
For $k>0$, this amounts to proving that
\begin{equation}\label{eq:intro_pairing_cycles_and_eisenstein}
  \brk{\mfZ_{\SL_2(\Z),2k}, \Eis_{2k}}
  =
  \frac{\Z}{\num(\zeta(-1-2k))},
\end{equation}
where $\Eis_{2k} \in H^1\!\left(\PSL_2(\Z), \LS_{2k}^{\vee} \otimes \Q\right)$ 
denotes the Eisenstein class (see~\cite[Section~2.7]{BS25}), 
and $\brk{\,\cdot\,,\,\cdot\,}$ is the natural pairing between homology and cohomology.
We first examined whether $\mfZ_{\SL_2(\Z),2k}$ coincides with the whole group
$H_1(\PSL_2(\Z), \LS_{2k})$. 
However, calculations in some concrete cases suggested that this is not the case in general. 
We therefore turned to studying the size of $\mfZ_{\SL_2(\Z),2k}$ in the
$p$-ordinary part of $H_1(\PSL_2(\Z), \LS_{2k})$, since the Eisenstein class
$\Eis_{2k}$ is $p$-ordinary at every prime number $p$. 
In the end, we established the above equality using the following theorem.

\begin{thm}[{\cite[Theorem 9.22]{BS25}}]\label{intro:N geq 5 and p mid N case fZ = H_1-ordinary}
For any prime number $p$ and any integer $N$ such that $p \mid N$ and $N \geq 4$, we have
\[
\mfZ_{\Gamma_{1}(N), 2k}^{\pord}  
= H_1^\ord(\overline{\Gamma_1(N)}, \LS_{2k, \bZ_{p}}).   
\]
Here, $\LS_{2k,\Z_p} := \LS_{2k}\otimes \Z_p$, and $\mfZ_{\Gamma_{1}(N), 2k}^{\pord}$ (resp. $H_1^\ord(\overline{\Gamma_1(N)}, \LS_{2k, \bZ_{p}})$) denotes the $p$-ordinary part of $\mfZ_{\Gamma_{1}(N), 2k}$ (resp. $H_1(\overline{\Gamma_1(N)}, \LS_{2k, \bZ_{p}})$). 
\end{thm}

Although Theorem~\ref{intro:N geq 5 and p mid N case fZ = H_1-ordinary} was sufficient for our purposes in \cite{BS25}, it naturally led us to ask whether this theorem extends to the case $p \nmid N$. 
The aim of this paper is to address this question. 
In Section~\ref{sec:ordinary part}, we prove the following theorem, which settles this problem completely.

\begin{thm}[Theorem \ref{thm:main}]\label{intro:thm_main}
Let $p$ be a prime number and let $N$ be a positive integer.
Let $\Gamma$ be a congruence subgroup such that $\Gamma_1(N) \subset \Gamma \subset \Gamma_0(N)$. 
Then  we have  
    \[
\mfZ_{\Gamma, 2k}^{\pord}  = H_1^\ord(\overline{\Gamma}, \LS_{2k} \otimes_{\bZ} \bZ_{p}). 
    \]
\end{thm}

\begin{rmk}
The equality \eqref{eq:intro_pairing_cycles_and_eisenstein} is proved in
\cite[Theorem~9.15]{BS25}.
Since the proof there relies on
Theorem~\ref{intro:N geq 5 and p mid N case fZ = H_1-ordinary}
(\cite[Theorem~9.22]{BS25}),
it requires several technical computations.
However,  using Theorem~\ref{intro:thm_main} instead of
Theorem~\ref{intro:N geq 5 and p mid N case fZ = H_1-ordinary}, we obtain a much simpler proof. 
\end{rmk}

As a consequence of Theorem~\ref{intro:thm_main} and the result of Goldman and Millson \cite{GM86}, we obtain the following theorem.

\begin{thm}[{Theorem \ref{thm:quotient_finite_non-ordinary}}]\label{intro:thm_non-ord_finite}
Let $N$ be a positive integer, and let $\Gamma$ be a congruence subgroup
satisfying $\Gamma_1(N) \subset \Gamma \subset \Gamma_0(N)$.
Then the quotient
\[
H_1(\overline{\Gamma}, \LS_{2k}) \big/ \mfZ_{\Gamma,2k}
\]
is a Hecke module over $\Z$ of finite order, and is non-ordinary at every prime $p$. 
\end{thm}

The arithmetic significance of $\mfZ_{\Gamma,2k}$ also motivates the study of the quotient $H_1(\overline{\Gamma}, \LS_{2k}) \big/ \mfZ_{\Gamma,2k}$ as a Hecke module. 
In ongoing work, the authors study the
case $\Gamma = \SL_2(\Z)$ and compute the structure of
the quotient $H_1(\overline{\Gamma}, \LS_{2k})/\mfZ_{\Gamma,2k}$ for various
integers $k$, establishing several properties of this quotient.

\subsection*{Acknowledgements}
Part of this research was carried out
during H.B.'s stay at the Max Planck Institute for Mathematics in Bonn. 
The authors are grateful to the institute for its hospitality and financial support. 
H.B. was supported by JSPS KAKENHI Grant Number JP25K23338 and Research Fellowship Promoting International Collaboration, The Mathematical Society of Japan. 
R.S. was supported by JSPS KAKENHI Grant Number JP24K16886.

\section{Preliminaries}

Let $\mathrm{PSL}_2(\bZ) := \mathrm{SL}_2(\bZ)/\{\pm 1\}$ be the projective special linear group over $\Z$. 
For any matrix $\gamma \in \mathrm{SL}_2(\bZ)$ (resp. subset $A \subset \mathrm{SL}_2(\bZ)$), we  denote by $\overline{\gamma} \in \mathrm{PSL}_2(\bZ)$ (resp. $\overline{A} \subset \mathrm{PSL}_2(\bZ)$) the image of $\gamma$ (resp. $A$) under the projection map $\mathrm{SL}_2(\bZ) \twoheadrightarrow  \mathrm{PSL}_2(\bZ)$.

\subsection{$M_2^+(\bZ)$-modules $\LS_k$}\label{subsection:Mn}

For any $2 \times 2$ matrix $\gamma=\begin{pmatrix}
    a&b\\
    c&d
\end{pmatrix}$, we denote the adjugate  of $\gamma$ by   
$\widetilde{\gamma}
:=
\begin{pmatrix}
    d&-b\\
    -c&a
\end{pmatrix}$. 
Put $M_2^+(\bZ) := \left\{\begin{pmatrix}
    a&b\\
    c&d
\end{pmatrix} \in M_2(\bZ) \, \middle| \, ad-bc > 0 \right\}$. 
In this paper, we define a left $M_2^+(\bZ)$-action on $\Hom(\bZ^2, \bZ)$ by 
\[
(\gamma  \phi)(x_1,x_2) := \phi((x_1, x_2) \cdot  {}^{t}\widetilde{\gamma}) = \phi(dx_1-bx_2, -cx_1+ax_2).  
\]
We also define
\[
\LS_k := \Sym^k(\Hom(\Z^2,\Z)), \qquad 
X_1 := \pr_1 \colon \bZ^2 \twoheadrightarrow \bZ,  \qquad X_2 := \pr_2 \colon \bZ^2 \twoheadrightarrow \bZ.
\]
Then the $M_2^+(\Z)$-module $\LS_k$ is canonically identified with the space of
homogeneous polynomials of degree $k$ in $\Z[X_1,X_2]$.
For notational simplicity, for any commutative ring $R$ we set
\[
\LS_{k,R} := \LS_k \otimes_{\Z} R.
\]

\subsection{Group homology}

Let $\Gamma$ be a (congruence) subgroup of $\mathrm{SL}_2(\bZ)$ and take  a left $\overline{\Gamma}$-module $\LS$. 
The $i$-th group homology $H_i(\overline{\Gamma},\LS)$ is defined to be the $i$-th Tor group $\mathrm{Tor}_i^{\bZ[\overline{\Gamma}]}(\bZ, \LS)$. 
Note that $H_0(\overline{\Gamma}, \LS) = \bZ \otimes_{\overline{\Gamma}} \LS$ is the module of
$\overline{\Gamma}$-coinvariants.

\subsubsection{$H_1(\overline{\Gamma},\LS)$}
Applying the tensor product functor $- \otimes_{\overline{\Gamma}} \LS$ to the short exact sequence of right $\bZ[\overline{\Gamma}]$-modules 
\[
0 \to I_{\overline{\Gamma}} \to \bZ[\overline{\Gamma}] \to \bZ \to 0, 
\]
where $I_{\overline{\Gamma}}$ denotes the augmentation ideal of $\bZ[\overline{\Gamma}]$, 
we obtain the exact sequence of $\bZ$-modules 
\[
0 \to H_1(\overline{\Gamma}, \LS) \to I_{\overline{\Gamma}}\otimes_{\overline{\Gamma}} \LS \to \LS \to H_0(\overline{\Gamma}, \LS) \to 0. 
\]
Consequently, the following natural identification is obtained:
\begin{align}\label{eq:identification_first_homology}
    H_1(\overline{\Gamma}, \LS) = \left\{\sum_{\gamma \in \overline{\Gamma}} (\gamma-1) \otimes v_\gamma  \in I_{\overline{\Gamma}} \otimes_{\overline{\Gamma}} \LS \,\, \middle| \,\,  \sum_{\gamma \in \overline{\Gamma}} (\gamma-1) v_\gamma  = 0  \right\}. 
\end{align}

\subsection{Definition of hyperbolic cycles}\label{sec:definition of hyperbolic cycles}

\begin{dfn}
    We say that a matrix $\gamma \in \mathrm{PSL}_{2}(\bZ)$ is hyperbolic if $|\mathrm{trace}(\gamma)| > 2$. 
For any subset $S \subset \mathrm{PSL}_{2}(\bZ)$, we denote by $S^{\mathrm{hyp}}$ the set of hyperbolic matrices in $S$. 
\end{dfn}

\begin{dfn}\label{def:rademacher integral}
For any hyperbolic matrix $\gamma :=
\begin{pmatrix}
    a&b\\
    c&d
\end{pmatrix}
\in \mathrm{PSL}_{2}(\bZ)^{\mathrm{hyp}}$, 
we define a binary quadratic form $Q_{\gamma}(X_1, X_2) \in \LS_2$ associated with $\gamma$ by
\[
Q_{\gamma}(X_1, X_2) := -\frac{\sgn(a+d)}{\mathrm{gcd}(c,a-d,b)}(cX_1^2-(a-d)X_1X_2-bX_2^2). 
\]
Note that $Q_\gamma(X_1,X_2)$ is $\gamma$-invariant. 
\end{dfn}

\begin{dfn}
Let $k \geq 0$ be an integer and let $\Gamma \subset  \mathrm{SL}_{2}(\bZ)$ be a congruence subgroup.  
For any matrix $\gamma \in \overline{\Gamma}^{\mathrm{hyp}}$, we define   
\[
\fz_{\Gamma, 2k}(\gamma) := (\gamma -1) \otimes Q_\gamma(X_1, X_2)^{k} \in H_1(\overline{\Gamma}, \LS_{2k}). 
\]
In this way, we obtain a map 
\begin{align*}
    \fz_{\Gamma, 2k} \colon \overline{\Gamma}^{\mathrm{hyp}} \to H_1(\overline{\Gamma}, \LS_{2k}).  
    \end{align*}
The homology class $\fz_{\Gamma,2k}(\gamma)$ is called the hyperbolic cycle of
degree $2k$ associated with $\gamma$.
We write $\mfZ_{\Gamma,2k}$ for the subgroup generated by the hyperbolic cycles
of degree $2k$, that is,
\[
\mfZ_{\Gamma,2k}
:=
\bigl\langle \,
  \fz_{\Gamma,2k}(\gamma)
  \bigm|
  \gamma \in \overline{\Gamma}^{\hyp} \,
\bigr\rangle_{\Z}
\subset
H_1(\overline{\Gamma}, \LS_{2k}).
\]

\end{dfn}

\begin{rmk}
In this paper, we changed the notation slightly from that of \cite{BS25}. The cycle $\mfz_{\Gamma,2k}(\gamma)$ (resp. the subgroup $\mfZ_{\Gamma,2k}$) in this paper is denoted by $\mfz_{\Gamma,k+1}(\gamma)$ (resp. $\mfZ_{\Gamma,k+1}$) in \cite{BS25}. 
\end{rmk}

\section{Hecke operators}

\subsection{Definition of  double coset operators}

Let $\Gamma $ and $\Gamma'$ be congruence subgroups of $\SL_2(\Z)$. 
Let $\LS$ be a left $M_2^+(\Z)/\{\pm 1\}$-module. 

\begin{dfn}
For any matrix $\alpha \in M_2^+(\Z)$, we define the double coset operator associated with $\alpha$
\[
 [\Gamma' \alpha \Gamma]  \colon H_i(\overline{\Gamma},\LS) \to H_i(\overline{\Gamma'},\LS)
\]
as follows. 
Let $\{\beta_1,\dots,\beta_r\}$ be a set of representatives for the quotient set $\overline{\Gamma'} \backslash \overline{\Gamma'} \alpha \overline{\Gamma}$.
Then we obtain a well-defined homomorphism
\begin{equation}\label{eqn:double coset operator in degree 0}
H_0(\overline{\Gamma}, \LS) \to H_0(\overline{\Gamma'}, \LS); \quad 
m \mapsto \sum_{i=1}^r \beta_i m.
\end{equation}
This homomorphism does not depend on the choice of  the set $\{\beta_1,\dots,\beta_r\}$ of representatives. 
By the theory of effaceable $\delta$-functors, the homomorphism \eqref{eqn:double coset operator in degree 0} naturally extends to a
homomorphism
\begin{equation}\label{eqn:double coset operator in general}
[\Gamma' \alpha \Gamma] \colon
H_i(\overline{\Gamma}, \LS)
\to
H_i(\overline{\Gamma'}, \LS).
\end{equation}    
\end{dfn} 
Put $\Gamma^\pm:=\Gamma\{\pm 1\}$ and $(\Gamma')^\pm:=\Gamma'\{\pm 1\}$. Set
\[
 \Gamma_1:=\Gamma^\pm\cap\alpha^{-1}(\Gamma')^\pm\alpha,
 \qquad
 \Gamma_2:=\alpha\Gamma^\pm\alpha^{-1}\cap(\Gamma')^\pm.
\]
Then there is a bijection 
\begin{align*}
\overline{\Gamma}_1\bs \overline{\Gamma} \stackrel{\sim}{\to} \overline{\Gamma'}\bs\overline{\Gamma'}\alpha\overline{\Gamma}; \quad 
\gamma \mapsto \alpha \gamma. 
\end{align*}
Hence, if $\{s_1,\dots, s_r\}$ is a set of representatives of
$\overline{\Gamma_1} \backslash \overline{\Gamma}$, the homomorphism
\eqref{eqn:double coset operator in degree 0} decomposes as the composite
homomorphism
\begin{align}\label{eq:decomp_double_coset_operator_degree_0}
\begin{array}{ccccccccc}
\displaystyle H_0(\overline{\Gamma}, \LS) 
& \stackrel{\Res}{\to} &
\displaystyle H_0(\overline{\Gamma}_1, \LS) 
& \stackrel{\alpha_*}{\to} &
\displaystyle H_0(\overline{\Gamma}_2, \LS) 
& \stackrel{\Cor}{\to} &
\displaystyle H_0(\overline{\Gamma'}, \LS) \\
m
& \mapsto &
\sum_{i=1}^r s_i m
& \mapsto &
\sum_{i=1}^r \alpha s_i m
& \mapsto &
\sum_{i=1}^r \alpha s_i m.  
\end{array}
\end{align}
Hence the same decomposition holds for the double coset operator
\eqref{eqn:double coset operator in general} in degree $i$, yielding
\begin{align*}
[\Gamma'\alpha \Gamma] \colon H_i(\overline{\Gamma},\LS)
\overset{\Res}{\to}
H_i(\overline{\Gamma}_1,\LS)
\overset{\alpha_*}{\to}
H_i(\overline{\Gamma}_2,\LS)
\overset{\Cor}{\to}
H_i(\overline{\Gamma'},\LS).   
\end{align*}

\subsubsection{$p$-th Hecke operator}

\begin{dfn}
Let $p$ be a prime number, let $N>0$ be an integer, and let $\Gamma$ be a
congruence subgroup satisfying $\Gamma_1(N)\subset\Gamma\subset\Gamma_0(N)$. 
We then define 
\begin{align*}
    T_p := [\,\Gamma \begin{pmatrix}
    1&0\\
    0&p
\end{pmatrix}\Gamma\,] 
\quad 
\textrm{if $p \nmid N$} \qquad \textrm{and} \qquad    
U_p := [\,\Gamma \begin{pmatrix}
    1&0\\
    0&p
\end{pmatrix}\Gamma\,]
\quad \textrm{if $p \mid N$}. 
\end{align*}
\end{dfn}

\begin{rmk}\label{rem:coset-decomposition}
When $p \nmid N$, take a matrix $\beta := \begin{pmatrix}m&n\\N&p\end{pmatrix} \in \Gamma_0(N)$. It is well-known that we have the following decompositions (see, for example, \cite[Propositions 3.33 and 3.36]{Shimura-book-arith-autom-funct}): 
\begin{align*}
\Gamma \begin{pmatrix}
    1& 0\\
    0&p
\end{pmatrix} \Gamma  =  
\begin{cases}
\bigsqcup_{j=0}^{p-1}
\Gamma 
\begin{pmatrix}
    1&j\\
    0&p
\end{pmatrix}  \sqcup   \Gamma  \beta\begin{pmatrix}
    p&0\\
    0&1
\end{pmatrix}
 & \textrm{if} \quad p\nmid N, 
\\ 
\\
\bigsqcup_{j=0}^{p-1}
\Gamma
\begin{pmatrix}
    1&j\\
    0&p
\end{pmatrix}  & \textrm{if} \quad  p\mid N. 
\end{cases}
\end{align*}
\end{rmk}

\subsection{Hecke stability of $\mfZ_{\Gamma, 2k}$}

\begin{prop}\label{prop:Hecke stability}
Let $\Gamma $ and $\Gamma'$ be congruence subgroups of $\SL_2(\Z)$. 
For any matrix $\alpha \in M_2^+(\Z)$, the subgroup $\mfZ_{\Gamma, 2k} \subset H_1(\overline{\Gamma},\LS_{2k})$ maps to $\mfZ_{\Gamma', 2k} \subset H_1(\overline{\Gamma'},\LS_{2k})$ under the double coset operator $[\Gamma' \alpha \Gamma]$. 
\end{prop}
\begin{proof}
Set $\Gamma_1:=\Gamma^\pm\cap\alpha^{-1}(\Gamma')^\pm\alpha$ and
$\Gamma_2:=\alpha\Gamma^\pm\alpha^{-1}\cap(\Gamma')^\pm$,
as in Section~3.1. 
Since  the double coset operator $[\Gamma' \alpha \Gamma] \colon H_1(\overline{\Gamma},\LS_{2k}) \to H_1(\overline{\Gamma'},\LS_{2k})$  decomposes as the composite homomorphism 
\begin{align*}
H_1(\overline{\Gamma},\LS_{2k})
\overset{\Res}{\ra}
H_1(\overline{\Gamma}_1,\LS_{2k})
\overset{\alpha_*}{\ra}
H_1(\overline{\Gamma}_2,\LS_{2k})
\overset{\Cor}{\ra}
H_1(\overline{\Gamma'},\LS_{2k}),  
\end{align*}
it suffices to prove that 
\begin{align*}
\Res(\mfZ_{\Gamma, 2k})\subset \mfZ_{\Gamma_{1}, 2k}, \qquad 
\alpha_*(\mfZ_{\Gamma_1, 2k})\subset \mfZ_{\Gamma_{2}, 2k}, \qquad
\Cor(\mfZ_{\Gamma_2, 2k})\subset \mfZ_{\Gamma', 2k}. 
\end{align*}
The inclusion $\Cor(\mfZ_{\Gamma_2, 2k})\subset \mfZ_{\Gamma', 2k}$ follows immediately from  the definition of $\Cor$.  
To show that $\alpha_*(\mfZ_{\Gamma_1,2k}) \subset \mfZ_{\Gamma_2,2k}$,
let $\gamma \in \Gamma_1$ be a hyperbolic matrix. 
One then has
\[
\alpha_* ( (\gamma - 1) \otimes Q_\gamma(X_1, X_2)^k ) 
= (\alpha \gamma \alpha^{-1} - 1) \otimes \alpha \cdot Q_\gamma(X_1, X_2)^k.
\]
Since  $\alpha \cdot Q_\gamma(X_1, X_2)$ is an 
$\alpha \gamma \alpha^{-1}$-invariant binary quadratic form with coefficients in $\Z$, it follows that 
\[
\alpha \cdot Q_\gamma(X_1, X_2) \in \Z Q_{\alpha \gamma \alpha^{-1}}(X_1, X_2).
\]
Hence, we obtain $\alpha_*(\mfZ_{\Gamma_1,2k}) \subset \mfZ_{\Gamma_2,2k}$ since $\alpha \Gamma_1 \alpha^{-1} = \Gamma_2$.

Let us show that $\Res(\mfZ_{\Gamma, 2k})\subset \mfZ_{\Gamma_{1}, 2k}$. 
We first consider how to compute the homology class  $\Res((\gamma-1)\otimes Q_{\gamma}(X_1, X_2)^k)$. 
Let $\LS$ be a left $\Z[\overline{\Gamma}]$-module. 
Since $\Z[\overline\Gamma]$ is free as a right $\Z[\overline{\Gamma}_1]$-module and $\bZ \otimes_{\overline{\Gamma}} \Z[\overline\Gamma] \otimes_{\overline{\Gamma}_1} \LS = \bZ \otimes_{\overline{\Gamma}_1} \LS$, 
there is a natural isomorphism 
\begin{align*}\label{eqn:another presentation of homology}
H_i(\overline{\Gamma}_1,\LS)
\cong H_i(\overline{\Gamma}, \bZ[\overline{\Gamma}] \otimes_{\overline{\Gamma}_1}  \LS)    
\end{align*} 
for any integer $i \geq 0$. 
Moreover, under the identification \eqref{eq:identification_first_homology}, 
the isomorphism 
\[
\psi \colon H_1(\overline{\Gamma}_1,\LS)
\stackrel{\sim}{\to} H_1(\overline{\Gamma}, \bZ[\overline{\Gamma}] \otimes_{\overline{\Gamma}_1}  \LS)
\]
is the homomorphism induced by 
the inclusion map $I_{\overline{\Gamma}_1} \hookrightarrow I_{\overline{\Gamma}}$. 
Let $\{s_1,\dots, s_r\}$ be a set of representatives of
$\overline{\Gamma}_1 \backslash \overline{\Gamma}$.
We then have a $\overline{\Gamma}$-equivariant homomorphism
\[
R \colon \LS \to \Z[\overline{\Gamma}] \otimes_{\overline{\Gamma}_1} \LS; \quad 
 m \mapsto \sum_{i=1}^r s_i^{-1} \otimes s_i m,
\]
which is independent of the choice of representatives $s_1,\dots,s_r$.
Since the induced homomorphism 
\[
R \colon H_0(\overline{\Gamma}, \LS) \to H_0(\overline{\Gamma}, \Z[\overline{\Gamma}] \otimes_{\overline{\Gamma}_1} \LS) = H_0(\overline{\Gamma}_1,  \LS) 
\]
coincides with the restriction map $\mathrm{Res}$ in \eqref{eq:decomp_double_coset_operator_degree_0}, it follows that 
\[
R_* = \psi \circ \mathrm{Res}  \colon H_1(\overline{\Gamma}, \LS) \to H_1(\overline{\Gamma}, \Z[\overline{\Gamma}] \otimes_{\overline{\Gamma}_1} \LS).  
\]
Thus for any  hyperbolic matrix $\gamma \in \overline{\Gamma}^{\mathrm{hyp}}$, we obtain 
\[
(\psi \circ \Res)((\gamma-1)\otimes Q_{\gamma}(X_1, X_2)^k) = \sum_{i=1}^r (\gamma-1)s_i^{-1}\otimes s_i  Q_{\gamma}(X_1, X_2)^k, 
\]
where this equality holds in $I_{\overline{\Gamma}}\otimes_{\overline{\Gamma}_1}\LS_{2k}$. 
For each integer $1 \leq i \leq r$, define $g_i \in \overline{\Gamma}_1$ and $\nu(i) \in \{1,\dots, r\}$ by the equation 
\begin{align*}
{s_i \gamma^{-1}} = {g_i^{-1} s_{\nu(i)}}. 
\end{align*}
Viewing $\nu$ as a permutation of $\{1,\dots,r\}$, let
\[
\{1,\dots,r\}=\bigsqcup_{\lambda=1}^m I_\lambda
\]
be the decomposition into the $\nu$-orbits. 
We then have
\begin{align*}
\sum_{i=1}^r (\gamma-1)s_i^{-1}\otimes s_iQ_{\gamma}(X_1, X_2)^k
=
\sum_{\lambda =1}^m \sum_{i \in I_{\lambda}} (\gamma-1)s_i^{-1}\otimes s_i Q_{\gamma}(X_1, X_2)^k, 
\end{align*}
and it suffices to show that  
\begin{align*}
\sum_{i \in I_{\lambda}} (\gamma-1)s_i^{-1}\otimes s_i Q_{\gamma}(X_1, X_2)^k 
\in \psi(\mfZ_{\Gamma_1,2k} )
\end{align*}
for each integer $1 \leq \lambda \leq m$.  
Let $i \in I_\lambda$. For any integer $j \geq 1$, we have 
\begin{align*}
(\gamma-1)s_{\nu^j(i)}^{-1} \otimes s_{\nu^j(i)} Q_{\gamma}(X_1, X_2)^k
&= 
(\gamma-1)s_{\nu^j(i)}^{-1}g_{\nu^{j-1}(i)} \otimes g_{\nu^{j-1}(i)}^{-1} s_{\nu^j(i)}Q_{\gamma}(X_1, X_2)^k\\
&=
(\gamma-1) \gamma s_{\nu^{j-1}(i)}^{-1} \otimes s_{\nu^{j-1}(i)}Q_{\gamma}(X_1, X_2)^k, 
\end{align*}
where the second equality follows from the facts that  $\gamma \cdot Q_{\gamma}(X_1, X_2) = Q_\gamma(X_1, X_2)$ and ${s_{\nu^{j-1}(i)} \gamma^{-1}} = {g_{\nu^{j-1}(i)}^{-1} s_{\nu^j(i)}}$. 
Repeating this computation, we obtain
\begin{align}\label{eq:induction_nu_reduction}
(\gamma-1)s_{\nu^j(i)}^{-1}\otimes s_{\nu^j(i)} Q_{\gamma}^k
=
(\gamma-1)\gamma^j s_{i}^{-1} \otimes s_{i}Q_{\gamma}(X_1, X_2)^k. 
\end{align}
Since $\nu^{|I_\lambda|}(i) = i$, it follows that 
\begin{align*}
g := s_{i} \gamma^{|I_\lambda|}s_{i}^{-1} = g_{\nu^{|I_\lambda|-1}(i)}\cdots g_{\nu(i)}g_{i} \in \overline{\Gamma}_1. 
\end{align*}
The binary quadratic form $s_{i}Q_{\gamma}$ is primitive and $g$-invariant, and  hence 
\begin{align}\label{eq:induction_(s-1)(g-1)Q=0}
s_{i}Q_{\gamma}(X_1,X_2) =\pm Q_{g}(X_1,X_2), \qquad
(s_i^{-1}-1)(g-1) \otimes s_iQ_\gamma(X_1,X_2)^k = 0.     
\end{align}
Therefore, \eqref{eq:induction_nu_reduction}  and \eqref{eq:induction_(s-1)(g-1)Q=0} imply that  
\begin{align*}
\sum_{i \in I_{\lambda}} 
(\gamma-1)s_i^{-1}\otimes s_i Q_{\gamma}(X_1, X_2)^k
&= (\gamma^{|I_\lambda|}-1)s_i^{-1} \otimes s_{i}Q_{\gamma}(X_1, X_2)^k
\\
&= s_i^{-1}(g - 1) \otimes s_{i}Q_{\gamma}(X_1, X_2)^k
\\
&= (g - 1) \otimes s_{i}Q_{\gamma}(X_1, X_2)^k
\\
&= \pm \psi(\fz_{\Gamma_1, 2k}(g)). 
\end{align*}
\end{proof}

\section{Ordinary part}\label{sec:ordinary part}

\subsection{Definition of the ordinary part}\label{sec:Hecke operator}

Let $p$ denote a prime number and let $\LS$ be a finitely generated $\bZ_p$-module with a left $M_2^+(\Z)/\{\pm 1\}$-action. 
For any congruence subgroup $\Gamma$, we  put 
\[
e_p := \lim_{m\to \infty}[\,\Gamma \begin{pmatrix}1 & \\ & p\end{pmatrix} \Gamma\,]^{m!} \in \mathrm{End}_{\bZ_p}(H_i(\overline{\Gamma}, \LS)). 
\]
The limit is taken in the $p$-adic topology. 
We define the ordinary part of the homology group by 
\[
 H_{i}^{\mathrm{ord}}(\overline{\Gamma}, \LS) :=  e_pH_{i}(\overline{\Gamma}, \LS). 
\]
Since $e_p^2 = e_p$, we obtain a direct sum decomposition
\[
 H_{i}(\overline{\Gamma}, \LS)  =  H_{i}^{\mathrm{ord}}(\overline{\Gamma}, \LS)  \oplus  (1-e_p)H_{i}(\overline{\Gamma}, \LS). 
\]

\begin{dfn}
We define a map 
\[
\fz_{\Gamma, 2k}^{\pord} \colon \overline{\Gamma}^{\mathrm{hyp}} \to H_1^{\mathrm{ord}}(\overline{\Gamma}, \LS_{2k, \bZ_p}); \quad \gamma \mapsto  e_p\fz_{\Gamma, 2k}(\gamma) 
\]
and also define a $\bZ_{p}$-submodule  $\mfZ_{\Gamma, 2k}^{\pord} \subset H_1^{\mathrm{ord}}(\overline{\Gamma}, \LS_{2k, \bZ_p})$   by 
\[
\mfZ_{\Gamma, 2k}^{\pord} := 
e_p(\mfZ_{\Gamma,2k}\otimes \Z_p)
=
\bigl\langle \, \fz_{\Gamma, 2k}^{\pord} (\gamma) \bigm| \gamma \in \overline{\Gamma}^{\mathrm{hyp}} \, \bigr\rangle_{\bZ_{p}}. 
\]

\end{dfn}

\subsection{Main result}

We now state the main result of this paper.

\begin{thm}\label{thm:main}
Let $p$ be a prime number and $N$ a positive integer. 
Let $\Gamma$ be a congruence subgroup satisfying $\Gamma_1(N) \subset \Gamma \subset \Gamma_0(N)$. 
Then  we have  
    \[
\mfZ_{\Gamma, 2k}^{\pord}  = H_1^\ord(\overline{\Gamma}, \LS_{2k, \bZ_{p}}). 
    \]
\end{thm}

The proof of Theorem \ref{thm:main} will be given in the following subsections (see Theorems \ref{thm:main-wild} and \ref{thm:main-tame}).

\subsection{The case where $p \mid N$}

Let $p$ be a prime number and let $N\geq1$ satisfy $p\mid N$.
For every subgroup $\Gamma_1(N)\subset\Gamma\subset\Gamma_0(N)$,
the subspace $\F_pX_2^{2k}\subset\LS_{2k,\F_p}$ is
$\overline{\Gamma}$-stable: a matrix $\smat{a&b\\c&d}\in\Gamma$
acts on $\F_pX_2^{2k}$ by multiplication by $a^{2k}\bmod p$.

\begin{prop}\label{prop:hida-redction-to-the-constant-sheaf-case}
For every subgroup $\Gamma_1(N)\subset\Gamma\subset\Gamma_0(N)$,
the inclusion $j \colon \F_pX_2^{2k} \hookrightarrow\LS_{2k,\F_p}$ induces
an isomorphism
\[
 j_*\colon H_1^{\ord}(\overline{\Gamma}, \F_pX_2^{2k})
 \stackrel{\sim}{\to}
 H_1^{\ord}(\overline{\Gamma},\LS_{2k,\F_p}).
\]
\end{prop}
\begin{proof}
This is the homological version of the argument in
\cite[Theorem 2 in \S7.2]{Hida93}.
The assertion is immediate when $k=0$.
For $k>0$, every $u\in\Z$ and every $P\in\LS_{2k,\F_p}$ satisfy
\[
 \pmat{1&u\\0&p}\cdot P(X_1,X_2)
 =P(-uX_2,X_2)\in\F_pX_2^{2k}.
\]
Consequently, every representative $\smat{1&u\\0&p}$ in the coset
decomposition of $U_p$ acts as zero on
$\LS_{2k,\F_p}/\F_pX_2^{2k}$.
The corresponding maps on the homology complexes are therefore zero,
so $U_p$ annihilates $H_i(\overline{\Gamma}, \LS_{2k,\F_p}/\F_pX_2^{2k})$ for $i\geq0$. 
Applying the ordinary projector $e_p$ to the exact sequence
\[
 H_2(\overline{\Gamma},\LS_{2k,\F_p}/\F_pX_2^{2k}) \to H_1(\overline{\Gamma},\F_pX_2^{2k}) \stackrel{j_*}{\to}  H_1(\overline{\Gamma},\LS_{2k,\F_p}) \to 
 H_1(\overline{\Gamma},\LS_{2k,\F_p}/\F_pX_2^{2k})
\]
proves the proposition.
\end{proof}

\begin{thm}\label{thm:main-wild}
Let $p$ be a prime number and $N$ a positive integer such that $p\mid N$.
For any congruence subgroup $\Gamma$ satisfying
$\Gamma_1(N)\subset\Gamma\subset\Gamma_0(N)$, we have
\[
 \mfZ_{\Gamma,2k}^{\pord}
 =H_1^{\ord}(\overline{\Gamma},\LS_{2k,\bZ_p}).
\]
\end{thm}
\begin{proof}
First take $\Gamma=\Gamma_1(N)$.
For a hyperbolic element
$\gamma\in\overline{\Gamma_1(N)}\setminus\overline{\Gamma(p)}$,
choose a representative $\smat{a&b\\c&d}\in\Gamma_1(N)$.
Then $c\equiv a-d\equiv0\pmod p$ and $b\not\equiv0\pmod p$.
In particular, $\gcd(c,a-d,b)$ is a $p$-adic unit, and
\[
 Q_\gamma(X_1,X_2)\bmod p\in\F_p^\times X_2^2.
\]
If \(j' \colon \mathbb F_p \cong \mathbb F_pX_2^{2k} \stackrel{j}{\to} \LS_{2k,\mathbb F_p}\) denotes the \(\Gamma_1(N)\)-homomorphism defined by \(b\mapsto bX_2^{2k}\),  then it follows that
\[
 j'_*(\fz_{\Gamma_1(N),0}(\gamma)\bmod p)
 \in\F_p^\times\cdot
       (\fz_{\Gamma_1(N),2k}(\gamma)\bmod p).
\]
By \cite[Lemma 9.19]{BS25},  hyperbolic elements in $\overline{\Gamma_1(N)}$ generate $\overline{\Gamma_1(N)}$. 
Since $H_1(\overline{\Gamma_1(N)},\mathbb F_p) \cong  \overline{\Gamma_1(N)}^{\mathrm{ab}}\otimes_{\mathbb Z}\mathbb F_p$, 
Proposition~\ref{prop:hida-redction-to-the-constant-sheaf-case} gives a surjection
\begin{equation}\label{eq:surj_ord}
 \mathfrak Z_{\Gamma_1(N),2k}^{\pord}\twoheadrightarrow
 H_1^{\mathrm{ord}}(\overline{\Gamma_1(N)}, \LS_{2k,\mathbb F_p}).
\end{equation}
Now let $\Gamma_1(N)\subset\Gamma\subset\Gamma_0(N)$.
The group $\overline{\Gamma_1(N)}$ is normal in $\overline{\Gamma}$
and acts trivially on $\F_pX_2^{2k}$. 
Put $G := \overline{\Gamma}/\overline{\Gamma_1(N)}$. 
The homology five-term exact sequence therefore gives
\[
 H_1(\overline{\Gamma_1(N)},\F_pX_2^{2k}) \to H_1(\overline{\Gamma},\F_pX_2^{2k}) \to  H_1(G,
          \F_pX_2^{2k})\to 0.
\]
Since the homomorphism $ H_1(\overline{\Gamma_1(N)},\F_pX_2^{2k}) \to H_1(\overline{\Gamma},\F_pX_2^{2k})$ is $U_p$-equivariant, the operator $U_p$ acts on $H_1(G, \F_pX_2^{2k})$. 
We claim that 
\[
U_p \cdot H_1(G,   \F_pX_2^{2k}) = 0. 
\]
To see this, use the representatives $\smat{1&u\\0&p}$ for $U_p$,
with $0\leq u<p$ (see Remark \ref{rem:coset-decomposition}). 
For \(\gamma=\smat{a&b\\c&d}\in\Gamma\) and each such \(u\), there exists a unique integer \(v\) with \(0\le v<p\) such that
$b+ud-(a+uc)v\equiv 0\pmod p$, since \(a+uc\equiv a\not\equiv 0\pmod p\).
Then
\[
 \pmat{1&u\\0&p}\gamma\pmat{1&v\\0&p}^{-1}
 =\pmat{a+uc&(b+ud-(a+uc)v)/p\\pc&d-cv}
 \in\Gamma 
\]
since  its diagonal entries agree with those of $\gamma$ modulo $N$. 
Its image in $G$ therefore equals the image of $\gamma$. 
Moreover, every $\smat{1&u\\0&p}$ acts as the identity on $\F_pX_2^{2k}$. 
Thus, the action of $U_p$ on $H_1(G,\mathbb F_p X_2^{2k})$ is given  by a sum of $p$ identity maps. Hence $U_p \cdot H_1(G,   \F_pX_2^{2k}) = 0$, and we obtain a surjection $ H_1^{\ord}(\overline{\Gamma_1(N)},\F_pX_2^{2k})
 \twoheadrightarrow H_1^{\ord}(\overline{\Gamma},\F_pX_2^{2k})$. 
It follows from Proposition~\ref{prop:hida-redction-to-the-constant-sheaf-case} that 
\[
 H_1^{\ord}(\overline{\Gamma_1(N)},\LS_{2k,\F_p})
 \twoheadrightarrow H_1^{\ord}(\overline{\Gamma},\LS_{2k,\F_p})
\]
is surjective. 
Combining this with \eqref{eq:surj_ord} and
Proposition~\ref{prop:Hecke stability} gives a surjection 
\begin{equation}\label{eq:surj_ord_2}
 \mfZ_{\Gamma,2k}^{\pord}\twoheadrightarrow
 H_1^{\ord}(\overline{\Gamma},\LS_{2k,\F_p}).
\end{equation}
Finally, since the natural homomorphism $ H_1(\overline{\Gamma},\LS_{2k,\Z_p})\otimes_{\Z_p}\F_p
 \hookrightarrow H_1(\overline{\Gamma},\LS_{2k,\F_p})$ is injective, \eqref{eq:surj_ord_2} shows that $\mfZ_{\Gamma,2k}^{\pord} = H_1^{\ord}(\overline{\Gamma},\LS_{2k,\Z_p})$. 
\end{proof}

\begin{cor}\label{cor:isom_ord_reduction}
Let $p$ be a prime number and $N$ a positive integer such that $p\mid N$.
For any congruence subgroup $\Gamma$ satisfying
$\Gamma_1(N)\subset\Gamma\subset\Gamma_0(N)$, the  homomorphism  $ H_1^\ord(\overline{\Gamma},\LS_{2k,\Z_p})\otimes_{\Z_p}\F_p
 \stackrel{\sim}{\to}
 H_1^\ord(\overline{\Gamma},\LS_{2k,\F_p})$ is an isomorphism. 
\end{cor}
\begin{proof}
Injectivity is immediate, while surjectivity follows from \eqref{eq:surj_ord_2}. This is also a standard fact in Hida theory; see, for example, \cite[proof of Theorem~2 in \S 7.2]{Hida93}.
\end{proof}

\subsection{The case where $p \nmid N$}

Let $p$ be a prime number and $N \geq 1$ an integer satisfying $p \nmid N$. 
Let $\Gamma$ be a congruence subgroup satisfying $\Gamma_1(N) \subset \Gamma \subset \Gamma_0(N)$. 
For notational simplicity, we put 
\[
\Gamma' := \Gamma \cap \Gamma_0(p). 
\]
Let $\pi \colon H_1(\overline{\Gamma'},\LS_{2k,\mathbb Z_p}) \to H_1(\overline{\Gamma},\LS_{2k,\mathbb Z_p})$ denote the canonical homomorphism. We also define
\[
\varphi \colon
H_1(\overline{\Gamma},\LS_{2k,\mathbb Z_p})
\to 
H_1(\overline{\Gamma'},\LS_{2k,\mathbb Z_p})
\]
by the double coset operator $\varphi
:=
[
\Gamma'
\begin{pmatrix}
1&0\\
0&p
\end{pmatrix}
\Gamma]$. 
We use the same notation \(\pi\) and \(\varphi\) for the corresponding homomorphisms with coefficient module \(\LS_{2k,\mathbb F_p}\).
In the following, we compute the composites $\pi \circ \varphi$ and $\varphi \circ \pi$.
Since the homomorphism  $\pi$ coincides with the double coset operator 
$[\Gamma\begin{pmatrix}1&0\\0&1\end{pmatrix}\Gamma']$, 
both compositions  $\pi \circ \varphi$ and $\varphi \circ \pi$ can be evaluated by applying the usual rules for composing double coset operators (see, for example, \cite[page 52]{Shimura-book-arith-autom-funct}).

\begin{lem}\label{lem:hecke1} 
We have $T_p = \pi  \circ \varphi$ as double coset operators. 
\end{lem}
\begin{proof}
Take a matrix $\beta := \begin{pmatrix}m&n\\N&p\end{pmatrix} \in \Gamma_0(N)$.
A direct computation gives the following coset decomposition: 
\begin{align}\label{eq:coset-decomposition}
    \Gamma' \begin{pmatrix} 1 & 0 \\ 0 & p \end{pmatrix}\Gamma = \bigsqcup_{j=0}^{p-1} \Gamma' 
\begin{pmatrix}
    1&j\\
    0&p
\end{pmatrix}  \sqcup   \Gamma'  \beta\begin{pmatrix}
    p&0\\
    0&1
\end{pmatrix}.      
\end{align}
Hence the natural map 
$ \overline{\Gamma'}\backslash\overline{\Gamma'}
 \pmat{1&0\\0&p}\overline{\Gamma}
 \to
 \overline{\Gamma}\backslash\overline{\Gamma}
 \pmat{1&0\\0&p}\overline{\Gamma};   \overline{\Gamma'}\alpha \mapsto  \overline{\Gamma}\alpha$
is a bijection. 
From this fact, the lemma follows.  
\end{proof}

\begin{lem}\label{lem:hecke2}
Take \(\beta=\smat{m&n\\N&p}\in\Gamma_0(N)\) and set $V := [\Gamma'\beta\pmat{p&0\\0&1}\Gamma']$. 
The following identities hold. 
\begin{itemize}
\item[(1)] One has $\varphi \circ \pi=U_p+V$ as double coset operators.

\item[(2)] When $k>0$, we have  $V^2=0$ and $VU_p=0$ as endomorphisms of  $H_i(\overline{\Gamma'}, \LS_{2k,\F_p})$.

\item[(3)] The following identities hold as homomorphisms $H_i(\overline{\Gamma'},\LS_{2k,\mathbb Z_p}) \to H_i(\overline{\Gamma},\LS_{2k,\mathbb Z_p})$:
\[
 \pi VU_p=p^{2k+1}[\Gamma\beta\Gamma]\pi,
 \qquad
 T_p\pi U_p=\pi U_p^2+p^{2k+1}[\Gamma\beta\Gamma]\pi.
\]
\end{itemize}
\end{lem}
\begin{proof}
A direct computation shows that $\beta\smat{p&0\\0&1}$ normalizes $\Gamma'$.
It follows that 
\[
 \Gamma'\beta\pmat{p&0\\0&1}\Gamma'
 =\Gamma'\beta\pmat{p&0\\0&1}.
\]
Hence the coset decomposition \eqref{eq:coset-decomposition}  gives (1). 
We have 
\[
 \left(\beta\pmat{p&0\\0&1}\right)^2
 =p\pmat{m^2p+nN&n(m+1)\\Np(m+1)&Nn+p}, \qquad 
 \pmat{m^2p+nN&n(m+1)\\Np(m+1)&Nn+p}\in\Gamma_0(Np).
\]
Furthermore,
\[
 \beta\pmat{p&0\\0&1}\pmat{1&u\\0&p}
 =p\beta\pmat{1&u\\0&1}\qquad(0\leq u<p).
\]
Thus, every representative occurring in the double-coset expressions for \(V^2\) and \(VU_p\) has all of its entries divisible by \(p\).  
These facts show (2). 

Since $\smat{1&u\\0&1}\in\Gamma$ and \(\beta\) normalizes \(\Gamma\), we have $\beta\smat{1&u\\0&1}\beta^{-1}\in\Gamma$, and hence 
$p\beta\smat{1&u\\0&1}
=
\beta\smat{1&u\\0&1}\beta^{-1}p\beta
\in \Gamma p\beta$. 
Thus, after applying \(\pi\), all \(p\) representatives occurring in the double-coset expression for \(VU_p\) give the same map on homology.
Since the scalar matrix \(pI_2\) acts on the coefficient module
\(\LS_{2k,\mathbb Z_p}\) as multiplication by \(p^{2k}\), we obtain 
\[
\pi VU_p
=
p\cdot p^{2k}[\Gamma\beta\Gamma]\pi
=
p^{2k+1}[\Gamma\beta\Gamma]\pi.
\]
Finally, $T_p\pi=\pi(U_p+V)$ by Lemma~\ref{lem:hecke1} and (1),
which implies the second identity in (3).
\end{proof}

\begin{prop}\label{prop:image-ordinary}
If $k>0$, then $\pi\bigl(H_1^{\ord}(\overline{\Gamma'},\LS_{2k,\F_p})\bigr)
 =H_1^{\ord}(\overline{\Gamma},\LS_{2k,\F_p})$.
\end{prop}
\begin{proof}
Since the index $[\overline{\Gamma}:\overline{\Gamma'}] = p+1$ is coprime to $p$, we have 
\[
\pi(H_1(\overline{\Gamma'},\LS_{2k,\F_p})) =  H_1(\overline{\Gamma},\LS_{2k,\F_p}). 
\]
By Lemma~\ref{lem:hecke2}(2), the operator $V$ vanishes on
$H_1^{\ord}(\overline{\Gamma'},\LS_{2k,\F_p})$,
since $U_p$ is invertible on the ordinary part. 
Thus $T_p\pi =\pi U_p$ on $H_1^{\ord}(\overline{\Gamma'},\LS_{2k,\F_p})$, which implies
\[
 \pi\bigl(H_1^{\ord}(\overline{\Gamma'},\LS_{2k,\F_p})\bigr)
 \subset H_1^{\ord}(\overline{\Gamma},\LS_{2k,\F_p}).
\]
On the other hand, Lemma~\ref{lem:hecke2}(2) gives
$(U_p+V)^r = U_p^{r-1}(U_p+V)$ for any $r\geq1$. 
Since $ H_1^{\ord}(\overline{\Gamma},\LS_{2k,\F_p}) =T_p^r(H_1(\overline{\Gamma},\LS_{2k,\F_p}))$ and 
$H_1^{\ord}(\overline{\Gamma'},\LS_{2k,\F_p}) = U_p^{r-1}(H_1(\overline{\Gamma'},\LS_{2k,\F_p}))$ for sufficiently large $r$, it follows that
\begin{align*}
 H_1^{\ord}(\overline{\Gamma},\LS_{2k,\F_p}) =\pi(U_p+V)^r(H_1(\overline{\Gamma'},\LS_{2k,\F_p})) \subset\pi U_p^{r-1}(H_1(\overline{\Gamma'},\LS_{2k,\F_p}))
 =\pi(H_1^{\ord}(\overline{\Gamma'},\LS_{2k,\F_p})).
\end{align*}
This proves the reverse inclusion.
\end{proof}

\begin{cor}\label{cor:image-ordinary}
If $k>0$, then 
$\pi(H_1^{\ord}(\overline{\Gamma'},\LS_{2k,\Z_p}))
 =H_1^{\ord}(\overline{\Gamma},\LS_{2k,\Z_p})$. 
\end{cor}
\begin{proof}

We first prove the  inclusion
\[
 \pi\bigl(H_1^{\ord}(\overline{\Gamma'},\LS_{2k,\mathbb Z_p})\bigr)
 \subset
 H_1^{\ord}(\overline{\Gamma},\LS_{2k,\mathbb Z_p}).
\]
Let \(e_p\) be the \(T_p\)-ordinary projector on
\(H_1(\overline{\Gamma},\LS_{2k,\mathbb Z_p})\), and put
\[
M' :=H_1^{\ord}(\overline{\Gamma'},\LS_{2k,\mathbb Z_p}),
\qquad
M :=(1-e_p)H_1(\overline{\Gamma},\LS_{2k,\mathbb Z_p})
\] 
for notational simplicity.  
Since the diamond operator \([\Gamma\beta\Gamma]\) commutes with \(T_p\), it preserves \(M\).
Set
\[
F:=(1-e_p)\pi|_{M'} \colon M' \to M. 
\]
Since \(U_p\) is invertible on \(M'\), Lemma~\ref{lem:hecke2}(3) gives
$F =
T_pFU_p^{-1} -
p^{2k+1}[\Gamma\beta\Gamma]FU_p^{-2}$. 
We claim that
\[
F(M')\subset p^sM
\]
for every \(s\geq0\). The case \(s=0\) is immediate. Suppose that
\(F(M')\subset p^sM\). Since $F =
T_pFU_p^{-1} -
p^{2k+1}[\Gamma\beta\Gamma]FU_p^{-2}$, we have  
$F\equiv T_pFU_p^{-1}\pmod{p^{s+1}M}$. 
It follows inductively that, for every \(r\geq1\),
\[
F \equiv T_p^rFU_p^{-r}\pmod{p^{s+1}M}.
\]
Since \(T_p\) acts topologically nilpotently on \(M\), we may choose \(r\) such that $T_p^r(M)\subset p^{s+1}M$. Hence  we obtain $T_p^rFU_p^{-r}(M')\subset p^{s+1}M$, and $F(M')\subset p^{s+1}M$, which proves the claim by induction on \(s\). 
Thus \(F=0\), i.e., $(1-e_p)\pi(M') = 0$. 
It follows that  $\pi(H_1^{\ord}(\overline{\Gamma'},\LS_{2k,\mathbb Z_p}))
\subset H_1^{\ord}(\overline{\Gamma},\LS_{2k,\mathbb Z_p})$. 

Since the canonical homomorphism $H_1^{\ord}(\overline{\Gamma},\LS_{2k,\mathbb Z_p}) \otimes_{\mathbb{Z}_p} \mathbb{F}_p \hookrightarrow H_1^{\ord}(\overline{\Gamma},\LS_{2k,\mathbb F_p})$  is injective, Corollary \ref{cor:isom_ord_reduction} and Proposition \ref{prop:image-ordinary} imply that $\pi(H_1^{\ord}(\overline{\Gamma'},\LS_{2k,\Z_p}))
 =H_1^{\ord}(\overline{\Gamma},\LS_{2k,\Z_p})$. 
\end{proof}

\begin{thm}\label{thm:main-tame}
Let $p$ be a prime number and $N$ a positive integer such that  $p \nmid N$. 
Let $\Gamma$ be a congruence subgroup satisfying $\Gamma_1(N) \subset \Gamma \subset \Gamma_0(N)$. 
Then  we have 
    \[
\mfZ_{\Gamma, 2k}^{\pord}  = H_1^\ord(\overline{\Gamma}, \LS_{2k, \bZ_{p}}). 
    \]
\end{thm}
\begin{proof}
When $k=0$, the group $\overline{\Gamma}$ is generated by its hyperbolic
elements. 
It follows that $\mfZ_{\Gamma,0}=H_1(\overline{\Gamma},\Z)$,
and the assertion follows. 
Now let $k>0$. The group $\Gamma'$ lies between $\Gamma_1(Np)$ and
$\Gamma_0(Np)$, so Theorem~\ref{thm:main-wild} and
Corollary~\ref{cor:image-ordinary} give
\begin{align*}
 H_1^{\ord}(\overline{\Gamma},\LS_{2k,\Z_p}) =\pi\bigl(H_1^{\ord}(\overline{\Gamma'},\LS_{2k,\Z_p})\bigr) =\pi(\mfZ_{\Gamma',2k}^{\pord})
 \subset\mfZ_{\Gamma,2k}\otimes_{\Z}\Z_p, 
\end{align*}
where the  inclusion follows from Proposition~\ref{prop:Hecke stability}. 
Thus we obtain $\mfZ_{\Gamma, 2k}^{\pord}  = H_1^\ord(\overline{\Gamma}, \LS_{2k, \bZ_{p}})$. 
\end{proof}

\section{On the boundary part}

\subsection{Definition of the boundary part}

Let $\Gamma \subset \SL_2(\bZ)$ be a congruence subgroup. 
Let $I_2$ denote the $2\times2$ identity matrix. 
For any parabolic element $\gamma \in \overline{\Gamma} \setminus \{\overline{I_2}\}$,
that is, $|\mathrm{trace}(\gamma)| = 2$, we define a binary quadratic form
$Q_\gamma(X_1, X_2) \in \LS_2$ in the same manner as in Definition \ref{def:rademacher integral}. 
We set 
\[
\fz_{\Gamma, 2k}(\gamma) := (\gamma -1) \otimes Q_\gamma(X_1,X_2)^{k} \in H_1(\overline{\Gamma}, \LS_{2k})
\]
for any non-negative integer $k$. 
We then define the boundary subgroup $H_1^\partial(\overline{\Gamma}, \LS_{2k})$
to be the subgroup generated by $\fz_{\Gamma, 2k}(\gamma)$ for all nonidentity parabolic elements $\gamma \in \overline{\Gamma}$: 
\[
H_1^{\partial}(\overline{\Gamma}, \LS_{2k}) := \bigl\langle \, \fz_{\Gamma, 2k}(\gamma) \bigm| \gamma \in \overline{\Gamma\setm \{\pm I_2\}} \textrm{ with } |\mathrm{trace}(\gamma)| =  2 \, \bigr\rangle_{\bZ}. 
\]
We remark that, from a geometric point of view, the boundary subgroup $H_1^\partial(\overline{\Gamma}, \LS_{2k})$ agrees with the subgroup arising from the boundary of the Borel--Serre compactification of the modular curve associated with  $\Gamma$. 
We also note that the boundary subgroups are compatible with the double coset operators: 

\begin{lem}
 Let $\Gamma $ and $\Gamma'$ be congruence subgroups of $\SL_2(\Z)$. 
For any matrix $\alpha \in M_2^+(\Z)$, the boundary subgroup $H_1^{\partial}(\overline{\Gamma}, \LS_{2k}) \subset H_1(\overline{\Gamma},\LS_{2k})$ maps to $H_1^{\partial}(\overline{\Gamma'}, \LS_{2k}) \subset H_1(\overline{\Gamma'},\LS_{2k})$ under the double coset operator $[\Gamma' \alpha \Gamma]$. 
\end{lem}
\begin{proof}
 Since the conjugate of a parabolic element is again parabolic, this lemma can be proved by exactly the same argument as in Proposition~\ref{prop:Hecke stability}. 
\end{proof}

\subsection{Structure of the boundary part}

Let $T := \begin{pmatrix}1&1\\0&1\end{pmatrix}$. 
As a Hecke module, the boundary subgroup has a particularly simple structure:

\begin{lem}\label{lem:boundary-hecke-generater}
Let $N$ be a positive integer. 
Let $\Gamma$ be a congruence subgroup satisfying $\Gamma(N) \subset \Gamma \subset \Gamma_1(N)$. 
Put $w := \min\{n > 0 \mid T^n \in \Gamma\}$. 
Then the boundary subgroup $H_1^{\partial}( \overline{\Gamma}, \LS_{2k})$ is generated by $\fz_{\Gamma, 2k}(T^w)$ as a module over the Hecke ring associated with $\Gamma$, that is, 
\[
H_1^{\partial}(\overline{\Gamma}, \LS_{2k}) = \bigl\langle  [\Gamma \alpha \Gamma] (\fz_{\Gamma, 2k}(T^w)) \bigm| \alpha \in M_2^+(\Z)  \bigr\rangle_{\mathbb{Z}}. 
\]
\end{lem}
\begin{proof}
    Let $\gamma \in \Gamma^\pm := \Gamma\{\pm I_2\}$ be a nonidentity parabolic matrix with $\operatorname{trace}(\gamma) = 2$. Every nonidentity parabolic element of $\overline{\Gamma}$ can be represented by a matrix in $\Gamma^\pm$ with trace $2$.  
 Since $\Gamma(N) \subset \Gamma \subset \Gamma_1(N)$, we have $\Gamma(N) T^{w\bZ}\{\pm I_2\} = \Gamma^{\pm}$.  As $\Gamma^\pm\subset\Gamma(w)\{\pm I_2\}$, there exists $\alpha \in \operatorname{SL}_{2}(\mathbb{Z})$ satisfying  $\alpha^{-1}\gamma\alpha\in T^{w\bZ}$.
By the identity $\fz_{\Gamma, 2k}(\gamma'^d) = \pm d \cdot \fz_{\Gamma, 2k}(\gamma')$ for every parabolic matrix $\gamma'\in\Gamma^\pm\setminus\{\pm I_2\}$ and every integer $d\ne0$, we may assume that $T^{w\bZ} \cap \alpha^{-1}\Gamma^\pm\alpha$ is generated by $\alpha^{-1}\gamma \alpha$ and that there is a positive integer $h$ satisfying $T^{wh} = \alpha^{-1}\gamma \alpha$. 
    We prove 
    \[
        \fz_{\Gamma, 2k}(\gamma) = \pm [\Gamma \alpha \Gamma](\fz_{\Gamma, 2k}(T^w)). 
    \]
    Note that $T^{w\mathbb{Z}}=T^{\mathbb{Z}}\cap\Gamma=T^{\mathbb{Z}}\cap\Gamma^\pm$. 
    It follows from $\Gamma(N) T^{w\bZ}\{\pm I_2\} = \Gamma^{\pm}$ that we obtain a canonical bijection 
    \[
    (T^{w\bZ} \cap \alpha^{-1}\Gamma^\pm\alpha) \bs T^{w\bZ} \stackrel{\sim}{\to} (\Gamma^\pm \cap \alpha^{-1}\Gamma^\pm\alpha) \bs \Gamma^\pm. 
    \]
    In particular, the set $\{ T^{wj} \mid 0 \leq j < h\}$ is a complete set of representatives of $(\Gamma^\pm \cap \alpha^{-1}\Gamma^\pm\alpha) \bs \Gamma^\pm$. 
    The subgroup $\Gamma^\pm\cap\alpha^{-1}\Gamma^\pm\alpha$ is the inverse image in $\Gamma^\pm$ of $\overline{\Gamma}\cap\alpha^{-1}\overline{\Gamma}\alpha$. Hence the matrices $T^{wj}$, $0\leq j<h$, project to the coset representatives used to define $[\Gamma\alpha\Gamma]$.
    The double coset formula gives 
    \begin{align*}
        [\Gamma \alpha \Gamma] \cdot \fz_{\Gamma, 2k}(T^w) 
        = \alpha_* \left( \sum_{j=0}^{h-1} T^{wj} (T^w-1) \otimes X_2^{2k} \right)
        = \alpha_*((T^{wh}-1) \otimes X_2^{2k})
        = \pm \fz_{\Gamma, 2k}(\gamma), 
    \end{align*}
    where we use $T^{wh} = \alpha^{-1}\gamma\alpha$. 
\end{proof}

The following proposition is an application of Theorem \ref{thm:main}.

\begin{prop}\label{prop:boundary-Gamma_1}
Let $N$ be a positive integer. 
Let $\Gamma$ be a congruence subgroup satisfying $\Gamma(N) \subset \Gamma \subset \Gamma_1(N)$. 
Then the boundary subgroup $H_1^{\partial}(\overline{\Gamma},\LS_{2k})$ is contained in
$\fZ_{\Gamma,2k}\otimes_{\Z}\Z[1/N]$. 
\end{prop}
\begin{proof}
Since   $\Gamma$ is a congruence subgroup satisfying $\Gamma(N) \subset \Gamma \subset \Gamma_1(N)$, for any parabolic matrix $\gamma \in \Gamma$, we have 
\[
\gamma^N \in \Gamma(N) \quad \textrm{ and } \quad \fz_{\Gamma, 2k}(\gamma^N) = N \cdot \fz_{\Gamma, 2k}(\gamma). 
\]
It follows that the canonical homomorphism 
\[
H_1^{\partial}(\overline{\Gamma(N)}, \LS_{2k}) \otimes_\bZ \bZ[1/N] \twoheadrightarrow H_1^{\partial}(\overline{\Gamma}, \LS_{2k}) \otimes_\bZ \bZ[1/N]
\]
is surjective. 
Hence we may assume that $\Gamma = \Gamma(N)$. 
Let $\alpha := \begin{pmatrix}N&0\\0&1\end{pmatrix} \in M_2^+(\bZ)$. 
Then $\Gamma_1(N^2) \subset \alpha^{-1} \Gamma \alpha \subset \Gamma_0(N^2)$ and $\alpha^{-1} \Gamma \alpha /\Gamma_1(N^2) \cong \ker((\bZ/N^2\bZ)^\times \twoheadrightarrow (\bZ/N\bZ)^\times)$. 
In particular, for any  $\gamma \in  \Gamma$, we have $\gamma^N \in \alpha\Gamma_1(N^2)\alpha^{-1}$. 
Hence the double coset operator $[\Gamma\alpha \Gamma_1(N^2)]$ induces  surjections 
\begin{align*}
 H_1(\overline{\Gamma_1(N^2)}, \LS_{2k})  \otimes_{\bZ} \bZ[1/N] &\twoheadrightarrow
 H_1(\overline{\Gamma}, \LS_{2k}) \otimes_{\bZ} \bZ[1/N], 
\\
H_1^{\partial}(\overline{\Gamma_1(N^2)}, \LS_{2k})  \otimes_{\bZ} \bZ[1/N] &\twoheadrightarrow H_1^{\partial}(\overline{\Gamma}, \LS_{2k})  \otimes_{\bZ} \bZ[1/N], 
\\
\fZ_{\Gamma_1(N^2), 2k}  \otimes_{\bZ} \bZ[1/N] &\twoheadrightarrow  
\fZ_{\Gamma, 2k}  \otimes_{\bZ} \bZ[1/N]. 
\end{align*}
Therefore it suffices to prove that 
\[
H_1^{\partial}(\overline{\Gamma_1(N^2)}, \LS_{2k})\subset \fZ_{\Gamma_1(N^2), 2k} \otimes_{\bZ} \bZ[1/N].
\]
Take a prime number $p$ with $p \nmid N$, and let us compute $T_p (\fz_{\Gamma_1(N^2), 2k}(T))$. 
Fix $\beta := \begin{pmatrix}
        m&n\\N^2 & p
    \end{pmatrix} \in \Gamma_0(N^2)$ as in  Remark \ref{rem:coset-decomposition}. 
    Set 
    \[
    \alpha := \begin{pmatrix}
        1 & 0\\0&p
    \end{pmatrix} \quad \textrm{and} \quad 
    \langle p \rangle := [\Gamma_1(N^2) \beta \Gamma_1(N^2)] = \beta_*.
    \]
Then it follows from Remark \ref{rem:coset-decomposition} and the proof of Proposition \ref{prop:Hecke stability} that 
\begin{align*}
    &T_p (\fz_{\Gamma_1(N^2), 2k}(T)) 
    \\
    &= \alpha_* \left(\sum_{j=0}^{p-1} (T-1)\begin{pmatrix}
        1&j\\0&1
    \end{pmatrix}^{-1} \otimes \begin{pmatrix}
        1&j\\0&1
    \end{pmatrix}X_2^{2k} + (T-1)(\alpha^{-1} \beta \widetilde{\alpha})^{-1} \otimes (\alpha^{-1} \beta \widetilde{\alpha}) \cdot X_2^{2k} \right)
    \\
    &= \alpha_* \left((T^p-1) \otimes X_2^{2k} + ((\alpha^{-1} \beta \widetilde{\alpha})T(\alpha^{-1} \beta \widetilde{\alpha})^{-1}-1) \otimes (\alpha^{-1} \beta \widetilde{\alpha}) \cdot X_2^{2k}\right)
    \\
    &= (T-1) \otimes X_2^{2k} + ( \beta T^p \beta^{-1}-1) \otimes  p^{2k}\beta   X_2^{2k}
    \\
    &= (1 + p^{2k+1} \langle p \rangle)\fz_{\Gamma_1(N^2), 2k}(T). 
\end{align*} 
Since the operators $T_p$ and $\langle p \rangle$ commute, we conclude that 
\[
\lim_{n \to \infty} T_p^{n!} (\fz_{\Gamma_1(N^2), 2k}(T)) = \lim_{n \to \infty}(1 + p^{2k+1} \langle p \rangle)^{n!}\fz_{\Gamma_1(N^2), 2k}(T) = \fz_{\Gamma_1(N^2), 2k}(T) . 
\]
Hence, we obtain from Theorem \ref{thm:main} that 
\[
\fz_{\Gamma_1(N^2), 2k}(T) \in \fZ_{\Gamma_1(N^2), 2k}^{\pord}. 
\]
Thus Proposition \ref{prop:Hecke stability} and Lemma \ref{lem:boundary-hecke-generater} imply 
\[
H_1^{\partial}(\overline{\Gamma_1(N^2)}, \LS_{2k}) \subset \fZ_{\Gamma_1(N^2), 2k} \otimes_{\bZ} \bZ_p. 
\]
Since $p$ is an arbitrary prime number with $p \nmid N$, we obtain $H_1^{\partial}(\overline{\Gamma_1(N^2)}, \LS_{2k}) \subset \fZ_{\Gamma_1(N^2), 2k} \otimes_{\bZ} \bZ[1/N]$. 
\end{proof}

\begin{cor}\label{cor:boundary-rational}
For any congruence subgroup $\Gamma$, the boundary subgroup 
$H_1^{\partial}(\overline{\Gamma},\LS_{2k})$  is contained in
$\fZ_{\Gamma,2k}\otimes_{\Z}\Q$. 
\end{cor}
\begin{proof}
Since     $\Gamma$ is a congruence subgroup, there is a positive integer $N$ such that $\Gamma(N) \subset \Gamma$. 
Let $d := [\Gamma \colon \Gamma(N)]$. Then for any parabolic matrix $\gamma \in \Gamma$, we have 
\[
\gamma^d \in \Gamma(N) \quad \textrm{ and } \quad \fz_{\Gamma, 2k}(\gamma^d) = d \cdot \fz_{\Gamma, 2k}(\gamma). 
\]
This fact implies that the canonical homomorphism $H_1^{\partial}(\overline{\Gamma(N)}, \LS_{2k}) \otimes_\bZ \bQ \twoheadrightarrow H_1^{\partial}(\overline{\Gamma}, \LS_{2k}) \otimes_\bZ \bQ$ is surjective. 
Therefore, this corollary follows from Proposition \ref{prop:boundary-Gamma_1}. 
\end{proof}

\begin{cor}\label{cor:boundary-Gamma_0}
    For any square-free positive integer $N$, the boundary subgroup $H_1^{\partial}(\overline{\Gamma_0(N)}, \LS_{2k})$ is contained in $\fZ_{\Gamma_0(N), 2k} \otimes_{\bZ} \bZ[1/N]$. 
\end{cor}
\begin{proof}
Since $N$ is square-free, every parabolic element of $\overline{\Gamma_0(N)}$ lies in fact in $\overline{\Gamma_1(N)}$. 
This fact implies that the canonical homomorphism $H_1^{\partial}(\overline{\Gamma_1(N)}, \LS_{2k}) \twoheadrightarrow H_1^{\partial}(\overline{\Gamma_0(N)}, \LS_{2k})$ is surjective. 
Hence this corollary follows from Proposition \ref{prop:boundary-Gamma_1}. 
\end{proof}

\subsection{On the Hecke module $H_1(\overline{\Gamma}, \LS_{2k})/\fZ_{\Gamma, 2k}$}

Combining Corollary \ref{cor:boundary-rational} with a result by Goldman and Millson \cite{GM86}, we obtain the following theorem. 

\begin{thm}\label{thm:quotient-torsion}
    For any  congruence subgroup $\Gamma$,  we have 
    \[
    \fZ_{\Gamma, 2k} \otimes_{\bZ} \bQ = H_1(\overline{\Gamma}, \LS_{2k}) \otimes_{\bZ} \bQ. 
    \]
\end{thm}

\begin{proof}
By \cite[Corollary~1 to Theorem~C]{GM86}, we have
\[
\bigl(H_1^{\partial}(\overline{\Gamma}, \LS_{2k}) + \fZ_{\Gamma, 2k}\bigr)
\otimes_{\bZ} \bQ
=
H_1(\overline{\Gamma}, \LS_{2k}) \otimes_{\bZ} \bQ.
\]
Moreover, Corollary~\ref{cor:boundary-rational} shows that
\[
H_1^{\partial}(\overline{\Gamma}, \LS_{2k}) \otimes_{\bZ} \bQ
\subset
\fZ_{\Gamma, 2k} \otimes_{\bZ} \bQ.
\]
Combining these, we conclude that $\fZ_{\Gamma, 2k} \otimes_{\bZ} \bQ = H_1(\overline{\Gamma}, \LS_{2k}) \otimes_{\bZ} \bQ$. 
\end{proof}

\begin{thm}\label{thm:quotient_finite_non-ordinary}
    Let $N$ be a positive integer, and let $\Gamma$ be a congruence subgroup
satisfying $\Gamma_1(N) \subset \Gamma \subset \Gamma_0(N)$.
 Then 
 \[
H_1(\overline{\Gamma}, \LS_{2k})/\mfZ_{\Gamma, 2k}
\]
is a Hecke module which is finite and non-ordinary at all prime numbers $p$. 
\end{thm}
\begin{proof}
This theorem follows from  Proposition~\ref{prop:Hecke stability} and Theorems \ref{thm:main} and \ref{thm:quotient-torsion}. 
\end{proof}

\bibliography{references}
\bibliographystyle{amsalpha}

\end{document}